\documentclass[a4paper, 10pt]{article}
\usepackage{authblk, setspace, subfig, caption}
  \usepackage{url}
 \usepackage{amsmath, amssymb,amscd} 
 \usepackage{enumerate} 
 \usepackage[active]{srcltx} 
 \usepackage{amsthm, leftidx}
  \usepackage{multirow, booktabs}
\usepackage{algorithm}
\usepackage{algpseudocode}

\usepackage{graphicx}

\usepackage{authblk}

\usepackage{amsthm} 

\theoremstyle{definition}
\newtheorem{defn}{Definition}
\newtheorem{rem}{Remark}
\newtheorem{ex}{Example}

\theoremstyle{plain}
\newtheorem{thm}{Theorem}
\newtheorem{lem}{Lemma}
\newtheorem{prop}{Proposition}
\newtheorem{cor}{Corollary}

\newcommand{\RR}{\mathbb{R}}

\newcommand{\NN}{\mathbb{N}}
\newcommand{\NNo}{\mathbb{N}\cup \{0\}}

\def\states{\mathcal X}
\def\state{x}
\def\cset{\mathcal M}
\def\allgambles{\mathcal L}

\usepackage[square,sort,comma,numbers]{natbib} 
\usepackage{hyperref}

 \newcommand{\low}[1]{{\underline{#1}}}
 \newcommand{\up}[1]{{\overline{#1}}}

\title{Perturbation bounds and degree of imprecision for uniquely convergent imprecise Markov chains}
\author[1]{Damjan \v{S}kulj \\ University of Ljubljana, Faculty of Social Sciences \\ Kardeljeva pl. 5, SI-1000 Ljubljana, Slovenia \\ \href{mailto:damjan.skulj@fdv.uni-lj.si}{\tt damjan.skulj@fdv.uni-lj.si} }

\begin{document}

\maketitle

\begin{abstract}
	The effect of perturbations of parameters for uniquely convergent imprecise Markov chains is studied. We provide the maximal distance between the distributions of original and perturbed chain and maximal degree of imprecision, given the  imprecision of the initial distribution. The bounds on the errors and degrees of imprecision are found for the distributions at finite time steps, and for the stationary distributions as well. 
	
	
\end{abstract}

\section{Introduction}
Markov chains depend on a large number of parameters whose values are often subject to uncertainty. In the long run even small changes in the initial distribution or transition probabilities may cause large deviations. Several approaches to cope with uncertainty and estimate its magnitude have therefore been developed.  Perturbation analysis gives estimates for the differences between probability distributions when the processes evolve in time (\cite{Mitrophanov2005, Kartashov1986, Kartashov1986a, rudolf2015perturbation}) or for the stationary distributions (\cite{cho2001comparison, mouhoubi2010new}) based on the differences in parameters (see also \cite{Mitrophanov2003, mitrophanov2005ergodicity} for the continuous time case). 

In the recent decades variety of models of \emph{imprecise probabilities}  \cite{augustin2014introduction} have been developed. They provide means of expressing probabilistic uncertainty in the way that no reference to particular precise models is needed. The results thus reflect exactly the amount of uncertainty that results from the uncertainty in the inputs. Uncertainty in the parameters of Markov chains has first been addressed with models of this kind by \citet{hart:98}, without formally connecting it to the theory of imprecise probabilities, although applying similar ideas. More recently, \citet{decooman-2008-a} formally linked Markov chains with the theory of upper previsions, while \citet{skulj:09} proposed an approach based on the theory of interval probabilities. Unique convergence of imprecise Markov chains has been further investigated by \citet{hermans2012characterisation}, where ergodicity of upper transition operators has been characterized in several ways, and by \citet{2013:skulj-hable}, who investigated the  generalization of coefficients of ergodicity. \citet{skulj:13b} also studied the structure of non-uniquely convergent imprecise Markov chains. 

Although at first glance models of imprecise probabilities and classical perturbation models seem to have the same objective, which they approach from different angles, this is often not so. In fact they answer essentially different questions. While imprecise probabilities provide models that replace classical probabilities with generalized models that are capable to reflect uncertainty or imprecision, perturbation models give the numerical information on how much uncertainty in results to expect, given the amount of uncertainty in inputs. Moreover, it is even quite natural to involve perturbation analysis to the models of imprecise probabilities. 

The goal of this paper is thus to apply results from perturbation theory to imprecise Markov chains. There are two main reasons for this. The first is, that imprecise models too, are sensitive to the changes of input parameters. In the case of imprecise Markov chains this means that the changes in the bounds of input models affect the bounds of the distributions at further times in a similar way as in the case with the precise models. The other, maybe even more important reason is that currently in the theory of imprecise probabilities there has been little attention paid to the 'degree of imprecision', that is the maximal distances between, say, lower and upper bounds of probabilities. At least in comparison with the attention received by the methods for calculating these bounds. We thus also estimate how the 'degree of imprecision' evolves in time for imprecise Markov chains. 

The paper is structured as follows. In Sc.~\ref{s-imc} we introduce imprecise Markov chains with a special emphasis on the representation of the probability distributions after a number of time steps. In Sc.~\ref{s-mpio} we introduce the metric properties of imprecise operators, which allow us to measure the distances between imprecise probability distributions. By the means of the distances we also define the degree of imprecision. We show that with exception of the special case of 2-alternating upper probabilities (or 2-monotone lower probabilities) it is hard to find the exact distance between two imprecise probability models. In Sc.~\ref{s-pio} we analyse the effects of perturbations of parameters to the deviations of the perturbed chains from the original ones. With a similar method we also study how the 'degree of imprecision' of the process grows in time. In Sc.~\ref{s-ex} we apply the analysis to the case of contamination models and give a numerical example.

\section{Imprecise Markov chains}\label{s-imc}
\subsection{Imprecise distributions and upper expectation functionals}
Let $\states$ be a finite set of \emph{states}. A \emph{probability 
distribution} of some random variable $X$ over $\states$ is given in terms of 
an \emph{expectation functional} $E$ on the space of real-valued maps $f$ on 
$\states$, which we will denote by $\allgambles(\states)$: 
\begin{equation}
 E(f) = \sum_{x\in \states} P(X = x) f(x). 
\end{equation}
An \emph{imprecise probability distribution} of a random variable $X$ is given 
in terms of a closed convex set of expectation functionals $\cset$, called a 
\emph{credal set}. To every imprecise probability distribution a unique 
\emph{upper expectation functional}  
\begin{equation}
 \up E(f) = \sup_{E\in \cset} E(f) 
\end{equation}
can be assigned. Upper expectation functionals are in a one-to-one correspondence with credal sets. Moreover, a functional $\up E$ is an upper expectation functional with respect to a credal set $\cset$ if and only if it satisfies the following properties: 
\begin{enumerate}[(i)]
	\item $\min_{\state\in\states} f(\state) \le \up E(f)\le \max_{\state\in\states} f(\state)$ (boundedness);
	\item $\up E(f_1+f_2) \le \up E(f_1) + \up E(f_2)$ (subadditivity);
	\item $\up E(\lambda f) = \lambda\up E(f)$ (non-negative homogeneity);
	\item $\up E(f + \mu 1_{\states}) = \up E(f) + \mu$ (constant additivity);
	\item if $f_1\le f_2$ then $\up E(f_1) \le \up E(f_2)$ (monotonicity),
\end{enumerate}
where $f, f_1, f_2\in \allgambles(\states)$ are arbitrary, $\lambda$ a 
non-negative real constant, $\mu$ an arbitrary real constant, and $1_{\states}$ 
the constant map 1 on $\states$. An upper expectation functional $\up E$ can be 
supplemented by a \emph{lower expectation functional} with respect to the same 
credal set by assigning:
\begin{equation}
 \low E(f) = \min_{E\in\cset} E(f). 
\end{equation}
The following duality relation holds:
\begin{equation}\label{eq-up-low-duality}
	\low E(f) = -\up E(-f)
\end{equation}
for every $f\in\allgambles(\states)$. A lower expectation functional $\low E$ 
satisfies the same properties (i)--(v) as the upper ones, except for (ii), where 
subadditivity is replaced by 
\begin{enumerate}[(ii)']
	\item $\low E(f_1+f_2) \ge \low E(f_1) + \low E(f_2)$ (superadditivity).
\end{enumerate}

\subsection{Representations of uncertainty}
In the previous section we described upper (and lower) expectation functionals, which uniquely represent convex sets of probability distributions. In principle such functionals posses properties similar to those of precise expectation functionals corresponding to precise probability distributions. However, there is a huge difference when it comes to the ways of specifying a convex set of probabilities compared to specifying a single probability distribution. In the latter case we need, in the general case, a single probability density function, or in the case finite spaces, a probability mass function. In the case of state spaces of finite Markov chains, we thus need to specify the probability of each state. In contrast, even in the case of finite spaces there are in general infinitely many values needed to specify a convex set of probability measures, or in the case of convex polytopes, this number is finite but often large. 

Every convex polytope can be represented by specifying its extreme points or as 
an intersection of a set of half spaces separated by hyperplanes. The latter is 
far more natural and useful in the case of imprecise probabilities. This 
approach is behind many particular models, such as \emph{lower} and 
\textit{upper probabilities}, \emph{probability intervals} or \textit{coherent 
lower} and \textit{upper previsions} \citep{augustin2014introduction, miranda:07, troffaes2014lower}. The most general of those models are 
lower and upper previsions, which generalize all the other models, including 
lower and upper expectation functionals. 

Although most researchers of imprecise probabilities list properties for lower 
previsions, in the theory of imprecise Markov chains, upper previsions are more 
often used. In general, an upper prevision $\up P\colon \mathcal K\to \RR$ is a 
map on some set $\mathcal K$ -- not necessarily a vector space -- of measurable 
maps $\states\to \RR$. What is important for our present model is the following 
equivalent definition of coherence for finite probability spaces. 
\begin{defn}
	Let $\mathcal K\subseteq \allgambles(\states)$ be a set of real-valued maps 
	on a finite set $\states$ and $\up P\colon \mathcal K\to \RR$ a mapping 
	such that there exists a closed and convex set $\mathcal M$ of 
	(precise/linear) expectation functionals such that $\up P(f) = \max_{P\in 
	\mathcal M}P(f)$. Then $\up P$ is a coherent upper prevision on $\mathcal 
	K$.  
\end{defn} 
Note that any coherent upper prevision $\up P$ allows canonical extension to 
entire $\allgambles(\states)$ by defining
\begin{equation}
\up E(f)=\max_{P\in \mathcal M}P(f).
\end{equation}
This upper expectation functional is called the \emph{natural extension} of 
$\up P$, and is clearly itself too a coherent upper prevision on $\mathcal 
L(\states)$. Upper expectation functionals thus form a subclass of coherent 
upper previsions. 

Another subclass of imprecise probability models are \emph{lower} and 
\textit{upper probabilities}. A lower and upper probability $\low P$ and $\up P$ respectively are defined as a pair of real-valued maps on a class $\mathcal A$ of subsets of $\states$. For every 
$A\in\mathcal A$, its probability is assumed to lie within $[\low P(A), \up 
P(A)]$. They too allow the formation of a credal set $\cset$, whose members are 
exactly those expectation functionals $E$ that satisfy the conditions $E(1_A) \ge \low P(A)$ 
and $E(1_A)\le \up P(A)$ for every $A\in \mathcal A$. If the bounds $\low P(A)$ 
and $\up P(A)$ are reachable by the members of $\cset$, then the lower/upper 
probabilities are said to be \emph{coherent}.

 When the lower and upper probabilities are defined on the set of elementary events 
 $A= \{x\}$, where $x\in\states$, we are talking about \textit{probability 
 intervals (PRI)}, which, if coherent, also form a subclass of coherent upper 
 previsions. 

When an upper probability $\up P$ satisfies the following equation:
\begin{equation}
 \up P(A \cup B) + \up P(A\cap B) \le \up P(A) + \up P(B) 
\end{equation}
we call it 2-\emph{alternating}. At the same time, the corresponding lower probability $\low P$, which satisfies equation $\low P(A) = 1-\up P(A^c)$, is 2-\emph{monotone}:
\begin{equation}
  \low P(A \cup B) + \low P(A\cap B) \ge \low P(A) + \low P(B). 
\end{equation}
If an imprecise probability model is given in the form of probability intervals $[\low p(x), \up p(x)]$ for all elements $x\in\states$, we can first extend the bounds to all subsets of $\states$ by 
\begin{equation}
\low P(A) = \max\left\{\sum_{x\in A} \underline p(x), 1-\sum_{x\in A^c} \overline p(x) \right\}, \text{ for every } A\subseteq \states.
\end{equation}
It can be shown that in this case $\low P$ is 2-monotone (see e.g. \cite{decampos:94}). 

\subsection{Calculating the natural extension as a linear programming problem} 
Given a (coherent)\footnote{Coherence is in fact not important for calculating the natural extension.} upper prevision $\up P$ on a set $\mathcal K$, its natural extension 
to $\allgambles(\states)$ can be calculated as a linear programming problem as 
follows:
\begin{quote}
Maximize
\begin{align}
	E(f) & = \sum_{x\in\states} p(x)f(x)
	\intertext{subject to} 
	\sum_{x\in\states} p(x) & = 1 \\
	\sum_{x\in\states} p(x)h(x) & \le \up P(h) 
	\end{align}
	for every $h\in \mathcal K$. 		
\end{quote} 
Although there exist efficient algorithms for solving linear programming 
problems, in the case of 2-monotone lower or 2-alternating upper probabilities 
the expectation bounds can be calculated even more efficiently by the use of 
\emph{Choquet integral} (see e.g.\cite{den:97}): 
\begin{align}
	\up E(f) & = \min f + \int_{\min f}^{\max f} \up P(f\ge x) dx 
	\intertext{and}
	\low E(f) & = \min f + \int_{\min f}^{\max f}  \low P(f\ge x)dx
\end{align} 

\subsection{Imprecise transition operators}
An \emph{(imprecise) Markov chain} $\{ X_n \}_{n\in \NNo}$ with the set of states $\states$ is specified by an initial (imprecise) distribution and an (imprecise) transition operator. A \emph{transition operator} assigns a probability distribution of $X_{n+1}$ conditional on $(X_n = x)$. As before, we represent the conditional distribution with a conditional expectation functional $T(\cdot|x)$, mapping a real-valued map on the set of states into $T(f|x)$. A transition operator $T\colon \mathcal L(\states)\to \mathcal L(\states)$ then maps $f\mapsto Tf$ such that
\begin{equation}
 Tf(x) = T(f|x). 
\end{equation}
That is, the value of $Tf(x)$ equals the conditional expectation of $f$ at time $n+1$ if the chain is in $x$ at time $n$. 

Replacing precise expectations $T(\cdot|x)$ with the imprecise ones given in terms of upper expectation functionals $\up T(\cdot|x)$, we obtain an imprecise transition operator $\up T\colon \mathcal L(\states)\to \mathcal L(\states)$ defined with 
\begin{equation}
\up Tf(x) = \up T(f|x).
\end{equation}
To every conditional upper expectation functional $\up T(\cdot | x)$ a credal set can be assigned, and therefore a set of transition operators $\mathcal T$ can be formed so that $\up Tf = \max_{T\in \mathcal T}Tf$. 

Given an upper expectation functional $\up E_0$ corresponding to the distribution of $X_0$ and an upper transition operator $\up T$, we obtain 
\begin{equation}
 \up E_n (f) = \up E_0(\up T^nf), 
\end{equation}
where $\up E_n$ is the upper expectation functional corresponding to the distribution of $X_n$. 
\begin{ex}\label{ex-1}
	Let an imprecise transition operator be given in terms of a lower and upper transition matrices:
	\begin{align}
	\low M = 
	\begin{bmatrix}
	0.33 & 0.33 & 0 \\
	0.33 & 0.17 & 0.25 \\
	0 & 0.5 & 0.42 
	\end{bmatrix} \qquad \text{and} \qquad 
	\up M = 
	\begin{bmatrix}
	0.67 & 0.67 & 0 \\
	0.58 & 0.42 & 0.5 \\
	0 & 0.58 & 0.5 
	\end{bmatrix}
	\end{align}
	The upper transition operator is then 
	\begin{equation}\label{ex-1-lp-T}
	\up Tf = \max_{\substack{\low M \le M \le \up M\\\noalign{\smallskip} M1_\states = 1_\states}} Mf,
	\end{equation}
	Thus, $\up Tf$ is the maximum of $Mf$ over all row stochastic matrices that lie between $\low M$ and $\up M$. Since each row can be maximized separately, the componentwise maximal vector does exist. 
	
	Further let 
	\begin{equation}
	\low P_0 = (0.33, 0.25, 0.25) \qquad \text{and} \qquad  \up P_0 = (0.38, 0.38, 0.42) 
	\end{equation}
	be the bounds for initial probability mass functions, whence the initial upper expectation functional is defined with 
	\begin{equation}\label{ex-1-lp-E}
	\up E_0(f) = \max_{\substack{\low P_0 \le P \le \up P_0\\\noalign{\smallskip} P(1_\states) = 1}} P(f).
	\end{equation}
	Both $\up T$ and $\up E_0$ are operators whose values are obtained as solutions of linear programming problems. As we explained in previous sections, we would hardly expect $\up E_n$, that is the upper expectation functional for $X_n$, to be easily expressed in terms of a single linear programming problem, but rather as a sequence of problems. Thus, to obtain the value of $\up E_n(f)$, for some $f\in\mathcal L(\states)$, we would first find $\up Tf$ as a solution of the linear program \eqref{ex-1-lp-T} and then use it in the next instance of the linear program with the objective function replaced with the expectation of $\up Tf$ to obtain $\up T^2f$, until finally we would maximize $\up E_0(\up T^nf)$ as a linear program of the form \eqref{ex-1-lp-E}. Practically this means that even though $\up E_0$ is natural extension of a simple probability interval, the linear programs for $\up E_n$ are in general much more complex. A similar situation occurs with an $n$-step transition operator $\up T^n$. 
	
	Nevertheless, we might still be interested in the lower bounds for probabilities of events of the form $(X_n = x)$. We can, for instance, find the upper probability $\up P(X_n = x)$ as $\up E_0\up T^n 1_{\{x\}}$, understanding of course that the imprecise probability model for the distribution of $X_n$ is no longer a simple probability interval. 
	
	The above leads us to the idea, that even if the imprecise transition operators $\up T^n$ cannot be expressed in a simple form comparable to matrices, we might still want to provide the lower and upper transition matrix containing the information of the conditional probabilities $\low P(X_{m+n} = y|X_m = x)$
	and $\up P(X_{m+n} = y|X_m = x)$, again bearing in mind that this is not an exhaustive information on the imprecise transition model. The upper probability can, for instance, be calculated by finding $\up T^n1_{\{y\}}(x)$. 
	
	In our case, the lower and upper probability mass vectors for $X_3$ are 
	\begin{equation}
	\low P_3 = (0.1966, 0.2672, 0.1513) \qquad \text{and} \qquad  \up P_3 = (0.5293, 0.5799, 0.3903),
	\end{equation}
	and the lower and upper 3-step transition probabilities are 
	\begin{align}
	\low M_3 = 
	\begin{bmatrix}
	0.2195 & 0.2500 & 0.1040 \\
	0.2195 & 0.2583 & 0.1533 \\
	0.1650 & 0.3067 & 0.2205 
	\end{bmatrix} \qquad \text{and} \qquad 
	\up M_3 = 
	\begin{bmatrix}
	0.5898 & 0.5992 & 0.3350 \\
	0.5383 & 0.5730 & 0.4175 \\
	0.4239 & 0.5609 & 0.4175 
	\end{bmatrix}
	\end{align}
	
\end{ex}

\section{Metric properties of imprecise operators}\label{s-mpio}
\subsection{Distances between upper operators}
Let $\allgambles_1 = \{ f\in\allgambles(\states) \colon 0 \le f(x) \le 1\, \forall x\in \states \}$. 
In \cite{2013:skulj-hable} the following distance between two upper expectation functionals $\up E$ and $\up E'$ is defined:  
\begin{equation}\label{eq-distance-functionals}
 d(\up E, \up E') = \max_{f\in \allgambles_1} |\up E(f) - \up E'(f)|. 
\end{equation}
When restricted to precise expectation functionals, the above distance coincides with the \emph{total variation distance} for probability measures:
\begin{equation}
d(P, Q) = \max_{A\in \mathcal F}|P(A)-Q(A)|,
\end{equation}
for two probability measures on an algebra $\mathcal F$.  
For real-valued maps on $\states$ we use the \emph{Chebyshev distance} 
\begin{equation}
d(f, g) = \max_{x\in\states} |f(x)-g(x)|,
\end{equation}
which we also extend to upper transition operators $\up T$ and $\up T'$:  
\begin{equation}
d(\up T, \up T') = \max_x d(\up T(\cdot |x), \up T'(\cdot |x)) = \max_{f\in \allgambles_1}d(\up Tf, \up T'f) .
\end{equation}
Subadditivity of the upper expectation functionals implies that 
\begin{equation}\label{eq-norm1-ue}
|\up E (f) - \up E(g)| \le |\up E (f-g)|\vee |\up E (g-f)| \le d(f, g) . 
\end{equation}
Hence, 
\begin{equation}
d(\up Tf, \up Tg) = \max_{x\in\states} |\up T(f|x) - \up T(g|x)| \le d(f, g)
\end{equation}
for every upper transition operator $\up T$.
Similarly, we have that
\begin{equation}\label{eq-dist-ET}
 d(\up E\, \up T, \up E\, \up T') = \max_{f\in\allgambles_1} |\up E( \up T f) - \up E( \up T' f) | \le \max_{f\in\allgambles_1} d( \up T f, \up T' f ) = d(\up T, \up T') . 
\end{equation}
It is easy to see that the distances defined using the corresponding lower expectation functionals and transition operators are the same as those where the upper expectations are used. 

In practical situations the upper expectations are usually the natural extensions of some coherent upper previsions defined on some subset $\mathcal K\subset\allgambles(\states)$. It would be therefore very useful if the differences between the values of two upper previsions on the elements of $\mathcal K$ would give some information on the distances between their natural extensions. In general, however, this does not seem to be possible. For an illustration we give the following simple example. 
\begin{ex}
	Let $\states$ be a set of 3 states, say $x, y, z$ and $\mathcal K = \{ f_1 = (0, 1, 0), f_2 = (0.1, 1, 0) \}.$ Then let two lower/upper previsions be given with the values:
	\begin{align*}
		\low P_1(f_1) & = 0.3 & \low P_2(f_1) &= 0.3 \\
		\up P_1(f_2) & = 0.305 & \up P_2(f_2) &= 0.306 
	\end{align*}
	Now let $\up E_1$ and $\up E_2$ be the corresponding natural extensions and $h = (1, 0.5, 0)$. Then we have that 
	\[ \up E_1(h) = 0.2 \qquad\text{and}\qquad \up E_2(h) = 0.21. \]
	Thus, although the maximal distance between $\up P_1$ and $\up P_2$ on $\mathcal K$ is only $0.001$, the distance $d(\up E_1, \up E_2)$ is at least $0.01$, which is 10 times larger. 
\end{ex}
Unlike the general case, in the case of 2-monotone lower probabilities the following holds (\cite{2013:skulj-hable}, Proposition~22):
\begin{prop}\label{prop-dist-2-alt}
	Let $\low P_1$ and $\low P_2$ be 2-monotone lower probabilities and $\low E_1$ and $\low E_2$ the corresponding lower expectation functionals. Then 
	\begin{equation}
		d(\low E_1, \low E_2) = \max_{A\subseteq \states} |\low P_1(A) - \low P_2(A)|. 
	\end{equation}
\end{prop}
The fact that $\low P_1(A)-\low P_2(A) = (1-\up P_1(A^c)) - (1-\up P_2(A^c)) = \up P_2(A^c)-\up P_1(A^c)$ and $d(\up E_1, \up E_2) = d(\low E_1, \low E_2)$ implies that the choice of either upper or lower functionals does not make any difference.  

\subsection{Distances between upper and lower operators}\label{ss-dulo}
Let $\cset_1$ and $\cset_2$ be two credal sets with the corresponding lower and upper expectation functionals denoted by $\up E_1, \low E_1$ and $\up E_2, \low E_2$ respectively. It has been shown in \cite{2013:skulj-hable} that the maximal distance between the elements of two credal sets can be expressed in terms of the distance between the corresponding expectation functionals:
\begin{align}
	\up d(\cset_1, \cset_2) & := \max_{E_1\in \cset_1, E_2\in \cset_2} d(E_1, E_2) \\
	& = \max_{f\in \allgambles_1}\max \{ \up E_1(f)-\low E_2(f), \up E_2(f) - \low E_1(f) \} 
	\intertext{Now, since $1-f\in \allgambles_1$ iff $f\in \allgambles_1$ and using $\up E_i (1-f) = 1-\low E_i(f)$, the above simplifies into:}
	& = \max_{f\in \allgambles_1} \up E_1(f)-\low E_2(f).
\end{align}
It follows that
\begin{equation}
 \up d(\cset, \cset) = \max_{f\in \allgambles_1} \up E(f)-\low E(f), 
\end{equation}
which could be regarded as a measure of imprecision  of a credal set.  
The above equalities justify the following definition of a distance between upper and lower expectation functionals. 
\begin{equation}
 d(\up E_1, \low E_2) = \max_{f\in \allgambles_1} \up E_1(f)-\low E_2(f), 
\end{equation}
and
\begin{equation}
 d(\up T, \low T) = \max_{x\in\states}d(\up T(\cdot | x), \low T(\cdot|x)). 
\end{equation}
The following proposition holds:
\begin{prop}
	Let $\up E_1$ and $\low E_2$ be a lower and an upper expectation functionals. Then 
	\begin{equation}
		 \max_{f\in \allgambles_1}\up E_1(f)-\low E_2(f) = \max_{A\subseteq \states} \up E_1(1_A) - \low E_2(1_A),  
	\end{equation}		
	where $1_A$ denotes the indicator function of set $A$. This implies that 
	\begin{equation}
		 d(\up E_1, \low E_2) = \max_{A\subseteq \states}\up E_1(1_A)-\low E_2(1_A). 
	\end{equation}			
\end{prop}
\begin{proof}
	Let $\cset_1$ and $\cset_2$ be the credal sets corresponding to $\up E_1$ and $\low E_2$. We have that 
	\begin{align}
		d(\up E_1, \low E_2) & = \max_{E_1\in \cset_1, E_2\in \cset_2} d(E_1, E_2) \\
		& = \max_{E_1\in \cset_1, E_2\in \cset_2} \max_{f\in\allgambles_1}E_1(f) - E_2(f). 
	\end{align} 
	For every (precise) expectation functional $E_i$ there exists some probability mass function $p_i$ so that $E_i(f) = \sum_{x\in\states} p_i(x)f(x)$. 
	Now let $A=\{ x\colon p_1(x) \ge p_2(x) \}$ and let $F=1_A$. For every $f\in\allgambles_1$ we have that 
	\begin{align}
		E_1(f) - E_2(f) & = \sum_{x\in\states} (p_1(x)-p_2(x))f(x) \\
		& \le \sum_{x\in\states} (p_1(x)-p_2(x))F(x) \\
		& = E_1(F) - E_2(F).
	\end{align} 
	Thus we have that the difference $E_1(f)-E_2(f)$ is always maximized by an indicator function. Hence, 
	\begin{align}
	d(\up E_1, \low E_2) & = \max_{E_1\in \cset_1, E_2\in \cset_2} \max_{A\subseteq\states}E_1(1_A) - E_2(1_A) \\
	 & = \max_{A\subseteq\states}\{ \max_{E_1\in \cset_1} E_1(1_A) - \min_{E_2\in \cset_2} E_2(1_A)\} \\
	 & = \max_{A\subseteq\states} \up E_1(1_A) - \low E_2(1_A).
	 \end{align} 
\end{proof}

\subsection{Coefficients of ergodicity}
\emph{Coefficients of ergodicity} measure the rate of convergence of Markov chains. Given a metric $d$ on the set of probability distributions, a coefficient of ergodicity is a real-valued map $\tau\colon T\mapsto \tau(T)$ with the property that 
\[ d(pT, qT) \le \tau(T)d(p, q), \]
where $p$ and $q$ are arbitrary probability mass functions. In the general form coefficients of ergodicity were defined by \citet{seneta1979coefficients}, while in the form where the total variation distance is used, it was introduced by \citet{dobrushin:56}. Given a stochastic matrix $P$, the value of $\tau(P)$ assuming the total variation distance, equals maximal distance between the rows of $P$:
\begin{equation}
	\tau (P) = \max_{i, j}d(P_i, P_j), 
\end{equation}
where $P_i$ and $P_j$ are the $i$-th and $j$-th rows of $P$ respectively. In the operators notation we would write
\begin{equation}
	\tau(T) = \max_{x, y}d(T(\cdot|x), T(\cdot|y)).
\end{equation}
The general definition clearly implies that: 
\begin{equation}
 d(pT^n, qT^n) \le \tau(T)^n d(p, q). 
\end{equation}
Coefficients of ergodicity are also called \emph{contraction coefficients}. Since transition operators are always non expanding, which means that $\tau(T)\le 1$, the case of interest is usually when $\tau(T)$ is strictly less than 1. In such case $\tau(T)^n$ tends to 0 as $n$ approaches infinity, which means that the distance $d(pT^n, qT^n)$ approaches 0. This means that the distance between probability distribution of random variables $X_n$ of the corresponding Markov chain is diminishing, or equivalently, that the distributions converge to a unique limit distribution. 

Often, despite $\tau(T)=1$, the value of $\tau(T^r)$ might be strictly less than 1, which is sufficient to guarantee unique convergence. In fact, a chain is uniquely convergent exactly if $\tau(T^r)<1$ for some positive integer $r$. 

Coefficients of ergodicity have been generalized for the case of imprecise Markov chains too. The first, so called \emph{uniform coefficient of ergodicity} has been introduced by Hartfiel~\cite{hart:98} as 
\begin{equation}
	\tau(\mathcal T) = \max_{T\in\mathcal T}\tau(T), 
\end{equation}
where $\tau$ is the coefficient of ergodicity based on the total variation distance. If $\tau(\mathcal T)<1$ this implies unique convergence of the corresponding Markov chains in the sense that every subset of the chains uniquely converges. This implies that the upper and lower expectations converge too, but in order to ensure unique convergence in the sense of expectation bounds weaker condition suffices. 

The \emph{weak coefficient of ergodicity} for imprecise Markov chains was defined in \cite{2013:skulj-hable} as
\begin{equation}
 \rho(\up T) = \max_{x, y\in\states} d(\up T(\cdot|x), \up T(\cdot| y)). 
\end{equation}
That is, it is equal to the maximal distance between its row  upper expectation functionals. The following properties hold:
\begin{enumerate}[{(i)}]
	\item $\rho(\up T\, \up S)\le \rho(\up T)\rho(\up S)$ for arbitrary upper transition operators $\up T$ and $\up S$; 
	\item $d(\up E_1 \, \up T, \up E_2 \, \up T) \le d(\up E_1, \up E_2 )\rho (\up T)$ for arbitrary upper expectation functionals $\up E_1, \up E_2$ and transition operator $\up T$. 
\end{enumerate}
The following theorem holds. 
\begin{thm}[\cite{2013:skulj-hable}~Theorem~21]
	Let $\up T$ be an imprecise transition operator corresponding to a Markov chain $\{X_n \}_{n\in \NN}$. Then the chain converges uniquely if and only if $\rho(\up T^r)<1$ for some integer $r>0$. 
\end{thm}
If $\up T$ is an upper transition operator such that $\up T(\cdot |x)$ is the natural extension of some 2-alternating upper probability, then it follows from Proposition~\ref{prop-dist-2-alt} that 
\begin{equation}
	\rho(\up T) = \max_{A\subseteq \states}\max_{x, y\in\states} |\up T(A|x)-\up T(A|y)|.
\end{equation}
Ergodicity coefficient can be applied to a pair of upper and lower expectation functionals as follows. 
\begin{prop}\label{pr-ergodicity-lu}
	Let $\up E_1$ and $\low E_2$ be an upper and lower expectation functionals, and $\up T$ an upper transition operator. Then:
	\begin{equation}
		 d(\up E_1 \up T, \low E_2 \up T) \le d(\up E_1 , \low E_2 )\rho(\up T).  
	\end{equation}
\end{prop}
\begin{proof}
	Denote by $\cset_1$ and $\cset_2$ the credal sets corresponding to $\up E_1$ and $\low E_2$ respectively. Then we have that:
	\begin{align}
	d(\up E_1 \up T, \low E_2 \up T) & = \max_{f\in \allgambles_1} \up E(\up Tf)-\low E(\up Tf) \\
	& = \max_{f\in \allgambles_1} \max_{E_1\in\cset_1, E_2\in \cset_2}E_1(\up Tf)-E_2(\up Tf) \\
	& = \max_{E_1\in\cset_1, E_2\in \cset_2}\max_{f\in \allgambles_1} E_1(\up Tf)-E_2(\up Tf) \\
	& \le  \max_{E_1\in\cset_1, E_2\in \cset_2}d(E_1, E_2)\rho(\up T) \\
	&=d(\up E_1 , \low E_2 )\rho(\up T).		
	\end{align}
\end{proof}

%

\section{Perturbations of imprecise Markov chains}\label{s-pio}
\subsection{Distances between imprecise distributions of perturbed Markov chains}\label{ss-dbid}
Suppose we have two imprecise Markov chains given by initial expectation functionals $\up E_0, \up E_0'$ and upper transition operators $\up T, \up T'$. The $n$-th step upper expectation functionals are then $\up E_n = \up E_0\,\up T^n$ and $\up E'_n =\up E_0'\,\up T'^n$ respectively. Our goal is to find the bounds on the distances between $\up E_n$ and $\up E'_n$ if the distances $d(\up E_0, \up E_0')$ and $d(\up T, \up T')$ are known. We will also assume that both chains are uniquely convergent with weak coefficients of ergodicity $\rho_n = \rho(\up T^n)$ and $\rho'_n = \rho(\up T'^n)$, so that $\lim_{n\to \infty}\rho_n  = 0$ and $\lim_{n\to \infty}\rho'_n  = 0$. The latter conditions are clearly necessary and sufficient for unique convergence. Moreover, we will give bounds on the distance between the limit distributions $\up E_\infty$ and $\up E'_\infty$. To do so we will follow the similar derivation of the bounds for the case of precise (but not necessarily finite state) Markov chains by \citet{Mitrophanov2005}. 

We will make use of the following proposition:
\begin{prop}\label{pr-diff-sum}
	The following equality holds for a pair of imprecise Markov chains and every $n\in \NNo$:
	\begin{equation}
		\up E_0 \up T^n - \up E'_0 \up T'^n = (\up E_0 \up T^n - \up E'_0 \up T^n) + \sum_{i=0}^{n-1} (\up E'_0 \up T'^i \up T\, \up T^{n-i-1} - \up E'_0 \up T'^i \up T' \up T'^{n-i-1}),
	\end{equation}
	and therefore, 
	\begin{equation}
		d(\up E_0 \up T^n, \up E'_0 \up T'^n) \le d(\up E_0 \up T^n, \up E'_0 \up T^n) + \sum_{i=0}^{n-1} d(\up E'_0 \up T'^i \up T\, \up T^{n-i-1}, \up E'_0 \up T'^i \up T' \up T'^{n-i-1}).
	\end{equation}
\end{prop}
\begin{thm}\label{thm-diff-err}
Denote $E_n = d(\up E_n, \up E'_n)$ and $D = d(\up T, \up T')$. The following inequality holds:	
\begin{equation}
	 E_n \le E_0 \rho_n + D\sum_{i=0}^{n-1}\rho_{i}. 
\end{equation}
\end{thm}
\begin{proof}
	Proposition~\ref{pr-diff-sum} implies that 
	\begin{align} 
		E_n & \le d(\up E_0 \up T^n, \up E'_0 \up T^n) + \sum_{i=0}^{n-1}d(\up E'_0 \up T'^i \up T\, \up T^{n-i-1}, \up E'_0 \up T'^i \up T' \up T^{n-i-1}) 
		\intertext{denote $\up E'_{i+1} = \up E'_0\up T'^{i+1}$ and $\up E^*_{i+1} = \up E'_0\up T'^{i}\up T$ }
		& \le E_0\rho_n +  \sum_{i=0}^{n-1}d(\up E^*_{i+1}  \up T^{n-i-1}, \up E'_{i+1} \up T^{n-i-1}) \\
		& \le E_0\rho_n +  \sum_{i=0}^{n-1}d(\up E^*_{i+1} , \up E'_{i+1})\rho_{n-i-1} \\
		& = E_0\rho_n +  \sum_{i=0}^{n-1}d(\up E'_{i}\up T , \up E'_{i}\up T')\rho_{n-i-1} \\
		\intertext{by \eqref{eq-dist-ET}}
		& \le E_0\rho_n +  \sum_{i=0}^{n-1}d(\up T , \up T')\rho_{n-i-1} \\
		& = E_0\rho_n +  D\sum_{i=0}^{n-1}\rho_{n-i-1} \\
		& = E_0\rho_n +  D\sum_{i=0}^{n-1}\rho_{i} .		
	\end{align}
\end{proof}

\begin{cor}\label{cor-op-err}
Denote $D_n  = d(\up T^n, \up T'^n )$ (that is $D_1 = D$). The following inequality then holds:
\begin{equation}
	 D_n \le D_1\sum_{i=0}^{n-1} \rho_i.
\end{equation}
\end{cor}
\begin{proof}
	We have:
	\begin{align}
		d(\up T^n, \up T'^n) & = \max_{f\in\allgambles_1}d(\up T^nf, \up T'^nf) \\
		& = \max_{f\in\allgambles_1}\max_{x\in \states}d(\up T(\up T^{n-1}f|x), \up T'(\up T'^{n-1}f|x)) \\
		& = \max_{f\in\allgambles_1}\max_{x\in \states}d(\up T(\cdot|x)\up T^{n-1}f, \up T'(\cdot|x)\up T'^{n-1}f) \\
		& = \max_{x\in \states}d(\up T(\cdot|x)\up T^{n-1}, \up T'(\cdot|x)\up T'^{n-1}) \\
		\intertext{by Theorem~\ref{thm-diff-err}}
		& \le \max_{x\in \states}d(\up T(\cdot|x), \up T'(\cdot|x)) \rho_{n-1} + D_1\sum_{i=0}^{n-2}\rho_i \\
		& = D_1 \rho_{n-1} + D_1\sum_{i=0}^{n-2}\rho_i = D_1\sum_{i=0}^{n-1}\rho_i.
	\end{align}
\end{proof}

\begin{lem}\label{lem-sums}
	Let $\up T$ be an upper transition operator such that $\rho(\up T^r)=\rho_r =: \rho < 1$ and let $n = kr+m$, where $m<r$. Then 
	\begin{equation}
		\sum_{i=0}^{n-1}\rho_{i} \le r\frac{1-\rho^k}{1-\rho} + m\rho^k
	\end{equation}
	and 
	\begin{equation}
		\sum_{i=0}^{\infty}\rho_{i} \le \frac{r}{1-\rho}.
	\end{equation}		
\end{lem}
\begin{proof}
	Clearly $\{\rho_n\}$ is a non-increasing sequence. Moreover, it follows directly from the definitions and monotonicity that 
	\begin{equation}\label{en-lem-sums-1}
		\rho_i \le \rho^{r\left[\frac ir \right]},
	\end{equation}
	 where $[\cdot]$ denotes the integer part. The required equations are now obtained by taking sums of the left and right hand sides of equation~\eqref{en-lem-sums-1}. 
\end{proof}
Let $E_\infty$ denote the distance between the limit distributions $\up E_\infty$ and $\up E'_\infty$. The following corollary is a direct consequence of the above results. 
\begin{cor}\label{cor-diff-err}
Using the notation from Theorem~\ref{thm-diff-err} and Lemma~\ref{lem-sums} we have the following inequalities
		\begin{align}\label{eq-lem-sums-1}
			E_n & \le E_0 \rho^{k} + D\left(r\frac{1-\rho^k}{1-\rho} + m\rho^k \right) \\ 
			E_\infty  & \le \frac{Dr}{1-\rho} \label{eq-lem-sums-2} \\
			\intertext{and}
			D_n & \le D_1 \left(r\frac{1-\rho^k}{1-\rho} + m\rho^k\right).
		\end{align}		
\end{cor}

\begin{rem}
	Notice that $D_\infty = E_\infty$ for every uniquely convergent Markov chain.  
\end{rem}

\subsection{Degree of imprecision}
Let $\{ X_n\}_{n\in\NN}$ be an imprecise Markov chain and let $\up E_n$ denote the upper expectation functionals corresponding to the imprecise distributions of $X_n$. As a measure of the degree of imprecision, we have suggested in Sc.~\ref{ss-dulo} $I_n = d(\up E_n, \low E_n)$. Our goal in this section is to find bounds on $I_n$ given the initial imprecision $I_0$ and the imprecision of the transition operator, given by $\hat I=d(\up T, \low T)$. 

Similarly as in Proposition~\ref{pr-diff-sum} we have the following. 
\begin{prop}\label{pr-sum-lu}
	Let $\up E_0$ and $\low E_0$ be a pair of lower an upper expectation functionals and $\up T$ and $\low T$ an upper and lower transition operators. Then we have that
	\begin{equation}
			\up E_0 \up T^n - \low E_0 \low T^n = \up E_0 \up T^n - \low E_0 \up T^n + \sum_{i=0}^{n-1} (\low E_0 \low T^i \up T \,\up T^{n-i-1} - \low E_0 \low T^i \low T \up T^{n-i-1}),
	\end{equation}
	and therefore 
	\begin{equation}
		d(\up E_0 \up T^n, \low E_0 \low T^n) = d(\up E_0 \up T^n, \low E_0 \up T^n) + \sum_{i=0}^{n-1} d (\low E_0 \low T^i \up T \,\up T^{n-i-1} , \low E_0 \low T^i \low T \up T^{n-i-1}).
	\end{equation}
	
\end{prop}
\begin{thm}\label{thm-diff-impr}
	Let  $I_n = d(\up E_n, \low E_n)$, $\hat I=d(\up T, \low T)$ and $\rho_n = \rho(\up T^n)$. 
	The following inequality holds:	
	\begin{equation}
	I_n \le I_0 \rho_n + \hat I\sum_{i=0}^{n-1}\rho_{i}. 
	\end{equation}	
\end{thm}
\begin{proof}
	It follows directly from Proposition~\ref*{pr-ergodicity-lu} that 
	\begin{equation}
		d(\up E_0 \up T^n, \low E_0 \up T^n) \le I_0\rho_n. 
	\end{equation}		
	Further we have that for every $0\le i \le n-1$
	\begin{align}
		d(\low E_0 \low T^i \up T\, \up T^{n-i-1}, \low E_0 \low T^i \low T \up T^{n-i-1}) &= \max_{f\in \allgambles_1}| \low E_0[\low T^i \up T\,\up T^{n-i-1} f] - \low E_0[\low T^i \low T\up T^{n-i-1} f] | 
		\intertext{which is by definition, and by replacing $\low E_i = \low E_0 \low T^i$ we have}
		& = \max_{f\in \allgambles_1}| \low E_i[\up T\,\up T^{n-i-1} f] - \low E_i[\low T\up T^{n-i-1} f] | 
		\intertext{by \eqref{eq-norm1-ue} }
		& = \max_{f\in \allgambles_1} d(\up T\,\up T^{n-i-1} f,  \low T\up T^{n-i-1} f)  \\
		& = \max_{f\in \allgambles_1}\max_{x\in\states} d(\up T(\up T^{n-i-1} f|x),  \low T(\up T^{n-i-1} f|x))  \\
		& = \max_{x\in\states} d(\up T(\cdot |x)\low T^{n-i-1},  \low T(\cdot |x)\low T^{n-i-1}) \\ 
		\intertext{by Proposition~\ref{pr-ergodicity-lu} }
		& = \max_{x\in\states} d(\up T(\cdot |x),  \low T(\cdot |x))\rho_{n-i-1} \\ 
		& = \hat I\rho_{n-i-1}.
	\end{align}
	Now the required inequality follows directly by combining the above inequalities.
\end{proof}
The following corollaries now follow immediately using similar reasoning as in the case of distances in Sc.~\ref{ss-dbid}. 
\begin{cor}
	Let $\up T$ be an upper transition operator and denote $\hat I_n = d(\up T^n, \low T^n)$. The following inequality holds:
	\[ \hat I_n \le \hat I_1\sum_{i=0}^{n-1}\rho_i. \] 
\end{cor}
\begin{cor}
Using the notation from Theorem~\ref{thm-diff-impr} and Lemma~\ref{lem-sums} we have the following inequalities
	\begin{align}\label{eq-lem-sums-3}
	I_n & \le I_0 \rho^{k} + \hat I\left(r\frac{1-\rho^k}{1-\rho} + m\rho^k \right) \\ 
	I_\infty  & \le \frac{\hat Ir}{1-\rho}  \\
	\intertext{and}
	\hat I_n & \le \hat I_1 \left(r\frac{1-\rho^k}{1-\rho} + m\rho^k\right).
	\end{align}		
\end{cor}
\begin{rem}
	It is again clear that for every uniquely convergent imprecise Markov chain $\hat I_\infty = I_\infty$. 
\end{rem}
%
%

\section{Examples}\label{s-ex}
\subsection{Contamination models}
Let $\up E$ be an upper expectation functional and $\varepsilon>0$. Then we consider the \emph{$\varepsilon$-contaminated} upper expectation functional 
	\begin{equation}
	 \up E_\varepsilon (f) = (1-\varepsilon)\up E (f) + \varepsilon f_{\mathrm{max}}, 
	\end{equation}	
where $f_{\mathrm{max}} = \max_{x\in\states} f(x) =: \up V(f)$. Note that $\up V$ is the upper expectation functional whose credal set consists of all expectation functionals on $\allgambles(\states)$. It is called the \emph{vacuous upper prevision}. The upper transition operator $\up T_V(f) = f_{\mathrm{max}} 1_\states$ is called the \emph{vacuous upper transition operator}. 

Being a convex combination of $\up E$ and the vacuous upper prevision $\up V, \up E_\varepsilon$ is itself also an upper prevision. Similarly we could define an $\varepsilon$-contaminated upper transition operator with 
	\begin{equation}
	\up T_\varepsilon f = (1-\varepsilon)\up Tf + \varepsilon 1_\states f_{\mathrm{max}} = (1-\varepsilon)\up Tf + \varepsilon \up T_Vf. 
	\end{equation}
Let $\rho = \rho(\up T)$. Then we can explicitly find the coefficients of ergodicity for the contaminated model. 
\begin{prop}\label{prop-eps-cont}
	Let $\up E$ be an upper expectation functional, $\up T$ an upper transition operator and $\up E_\varepsilon$ and $\up T_\varepsilon$ the corresponding $\varepsilon$-contaminated models. The following inequalities hold:
	\begin{enumerate}[{(i)}]
		\item $d(\up E, \up E_\varepsilon) = \varepsilon d(\up E, \up V)$, where $\up V$ is the vacuous upper prevision;
		\item $d(\up T, \up T_\varepsilon) = \varepsilon d(\up T, \up T_V)$, where $\up T_V$ is the vacuous upper transition operator;
		\item $d(\up E'_\varepsilon, \up E_\varepsilon) = (1-\varepsilon)d(\up E', \up E)$;
		\item $d(\up T_\varepsilon, \up T'_\varepsilon) = (1-\varepsilon)d(\up T, \up T')$;
		\item $\rho(\up T_\varepsilon) = (1-\varepsilon)\rho(\up T)$; 
		\item $d(\up E_\varepsilon, \low E_\varepsilon) = (1-\varepsilon)d(\up E, \low E) + \varepsilon$; 
		\item $\hat I(\up T_\varepsilon) = (1-\varepsilon)\hat I(\up T) + \varepsilon$. 
	\end{enumerate}
\end{prop}
\begin{proof}
	(i) follows directly from
	\begin{align}
		d(\up E, \up E_\varepsilon) & = \max_{f\in \allgambles_1} | \up E(f) - (1-\varepsilon)\up E(f) - \varepsilon \up V(f) | \\
		& = \max_{f\in \allgambles_1} | \varepsilon(\up E(f) - \up V(f)) | \\
		& = \varepsilon d(\up E, \up V)
	\end{align}
	and simply leads to (ii). To see (iii) we calculate
	\begin{align}
	d(\up E_\varepsilon, \up E'_\varepsilon) & = \max_{f\in \allgambles_1} | (1-\varepsilon)\up E(f) + \varepsilon \up V(f) - (1-\varepsilon)\up E'(f) - \varepsilon \up V(f) | \\
	& = \max_{f\in \allgambles_1} | \varepsilon(\up E(f) - \up E'(f)) | \\
	& = \varepsilon d(\up E, \up E').
	\end{align}
	(iv) and (v) are direct consequences of (iii) and the definitions. To see (vi) note that $\low V(f) = f_{\mathrm{min}}$ and therefore $\low E_\varepsilon(f) = (1-\varepsilon)\low E(f) + \varepsilon \low V(f)$. For some $f$ we then have
	\begin{align}
		\up E_\varepsilon(f) - \low E_\varepsilon(f) & = (1-\varepsilon) (\up E(f) - \low E(f)) + \varepsilon (f_{\mathrm{max}} - f_{\mathrm{min}})
	\end{align}
	Now suppose the maximal difference in the above expression is attained for some $f\in \allgambles_1$ and denote $\tilde f = \dfrac{f-f_{\mathrm{min}}}{f_{\mathrm{max}}-f_{\mathrm{min}}}$ which belongs to $\allgambles_1$ as well. 	
	It is directly verified that $\up E(\tilde f)-\low E(\tilde f) = \dfrac{1}{f_{\mathrm{max}}-f_{\mathrm{min}}}(\up E(f)-\low E(f))$. Thus, $f_{\mathrm{max}}-f_{\mathrm{min}} = 1$ must hold, because of maximality of $f$, and therefore (vi) easily follows.  
	
	(vii) is also a simple consequence of (vi). 
\end{proof}
\begin{thm}
	Let $\up E$ and $\up T$ be an upper expectation functional and an upper transition operator respectively, and $\up E_\varepsilon$ and $\up T_\varepsilon$ the corresponding $\varepsilon$-contaminated operators. Denote $E_n = d(\up E_{\varepsilon n}, \up E_n)$, where $\up E_{\varepsilon n} = \up E_{\varepsilon}\up T_\varepsilon^n$ and $\up E_{n} = \up E\,\up T^n$,  $\Delta_1 = d(\up E, \up V), \Delta_2 = d(\up T, \up T_V), \hat I = \hat I(\up T), \rho = \rho(\up T)$ and the imprecision of the contaminated chain by $I'_n = d(\up E_{\varepsilon n}, \low E_{\varepsilon n})$. 
	
	Then we have that $E_0 = \varepsilon \Delta_1, D=\varepsilon \Delta_2$ and $\hat I(\up T_\varepsilon) = (1-\varepsilon)\hat I(\up T) + \varepsilon$. The following inequalities hold:
	\begin{align}
		E_n & \le \varepsilon \Delta_1\rho^n(1-\varepsilon)^n + \varepsilon \Delta_2 \frac{1-\rho^n (1-\varepsilon)^n}{1-\rho(1-\varepsilon)}; \\
		E_\infty & \le \frac{\varepsilon \Delta_2}{1-\rho(1-\varepsilon)}; \\
		I'_n & \le ((1-\varepsilon)I_0+\varepsilon)(1-\varepsilon)^n\rho^n +  ((1-\varepsilon)\hat I + \varepsilon)\frac{1-\rho^n(1-\varepsilon)^n}{1-\rho (1-\varepsilon)}; \\
		I'_\infty & \le \frac{(1-\varepsilon)\hat I + \varepsilon}{1-\rho(1-\varepsilon)}.
	\end{align}	
\end{thm} 
\begin{proof}
	A direct consequence of Theorems~\ref{thm-diff-err} and \ref{thm-diff-impr}, Proposition~\ref{prop-eps-cont} and Corollary~\ref{cor-diff-err}.
\end{proof}

\subsection{Numerical example}
We again consider a Markov chain with the initial lower and upper probabilities and transition probabilities as in Example~\ref{ex-1}. We compare it with a perturbed chain whose lower and upper transition matrices are
\begin{align}
\low M' = 
\begin{bmatrix}
0.32 & 0.36 & 0 \\
0.36 & 0.19 & 0.24 \\
0 & 0.5 & 0.4 
\end{bmatrix} \qquad \text{and} \qquad 
\up M' = 
\begin{bmatrix}
0.64 & 0.68 & 0 \\
0.57 & 0.38 & 0.45 \\
0.04 & 0.56 & 0.46 
\end{bmatrix}
\end{align}
and the initial probability bounds are 
\begin{equation}
\low P'_0 = (0.32, 0.21, 0.28) \qquad \text{and} \qquad  \up P'_0 = (0.42, 0.38, 0.42) 
\end{equation}
Coefficients of ergodicity are $\rho(\up T) = 0.67$ and $\rho(\up T') = 0.60$, 
and the distance between initial imprecise probability models is $d(\up E_0, 
\up E'_0) = 0.0248$, and between transition operators $d(\up T, \up T') = 
0.05$. 

The maximal theoretically possible bounds $d(E_n, E'_n)$ can be obtained using 
Theorem~\ref{thm-diff-err}, with $E_0 = 0.0248, D=0.05$ and $\rho_n=\rho(\up 
T')^n = 0.60^n$. For comparison we have calculated lower and upper transition 
probability matrices and the distances based on these estimates. The results are listed in Figure~\ref{fig-distances-fnals}. The actual 
distances may be larger because the expectation functionals are not fully 
described by probability interval models (PRI).   
\begin{figure}[h!]
	\caption{Distances between $\up E_n$ and $\up E'_n$ based on PRI estimates 
	and their theoretical upper bounds.} \label{fig-distances-fnals} 
	\[
	\begin{array}{l|rrrr}
	\toprule
	n									& 1 & 2 & 3 & \infty \\\hline 
	\text{PRI-distance}					& 0.0248 & 0.0387 & 0.0429 & 0.0467 \\ 
	\text{maximal theoretical distance}	& 0.0740 & 0.0889 & 0.1034 & 0.1250 \\
	\bottomrule
	\end{array}
	\]

\end{figure}

In Figure~\ref{fig-distances-ops} the distances between the operators $\up T^n$ 
and $\up T'^n$ are given, together with their upper bounds calculated using 
Corollary~\ref{cor-op-err}.
\begin{figure}
	\caption{Distances between $\up T^n$ and $\up T'^n$ based on PRI estimates 
	and their theoretical upper bounds.} \label{fig-distances-ops} 
	\[
	\begin{array}{l|rrrr}
	\toprule
	n									& 2 & 3 & 4 & \infty \\\hline 
	\text{PRI-distance}					& 0.0454 & 0.0499 & 0.0484 & 0.0467 \\ 
	\text{maximal theoretical distance}	& 0.0800 & 0.0980 & 0.1088 & 0.1250 \\
	\bottomrule
	\end{array}
	\]
	
\end{figure}

\section{Conclusions and further work}
We have studied the impact of perturbations of initial imprecise probability distributions and transition operators of imprecise Markov chains on the deviations of distributions of the chain at further steps. The results show that stability of the distributions depends on the weak coefficient of ergodicity, which is consistent with the known results for precise Markov chains \cite{Mitrophanov2005}. By the same means we give the bounds on the degree of imprecision depending on the imprecision of initial distribution and transition operators. 

Our goal in the future is to extend the results to related models, such as continuous time imprecise Markov chains, hidden Markov models or semi-Markov models. 

\subsection*{Acknowledgement}
The author would like to thank Alexander Y. Mitrophanov for his helpful suggestions and discussions. 

\bibliography{refer.bib}

\end{document}